\documentclass[11pt]{amsart}

\usepackage{amsmath,amsthm,amsfonts,amssymb,enumitem}
\usepackage{verbatim,todonotes}
\usepackage[sort,compress]{cite}
\usepackage[margin=1in]{geometry}
\usepackage[bookmarks]{hyperref}
\usepackage[all,cmtip]{xy}

\newcommand{\set}[1]{\left\{ #1 \right\}}
\newcommand{\abs}[1]{\left| #1 \right|}
\newcommand{\wt}[1]{\widetilde{ #1}}

\newcommand{\wtalpha}{\wt{\alpha}}

\newcommand{\F}{\mathbb{F}}

\newcommand{\Z}{\mathbb{Z}}

\newcommand{\g}{\mathfrak{g}}

\newcommand{\fu}{\mathfrak{u}}

\DeclareMathOperator{\chr}{char}

\DeclareMathOperator{\Der}{Der}

\DeclareMathOperator{\Ext}{Ext}

\DeclareMathOperator{\opH}{H}
\DeclareMathOperator{\Hbul}{\opH^\bullet}

\DeclareMathOperator{\Hom}{Hom}

\DeclareMathOperator{\Inn}{Inn}
\DeclareMathOperator{\Lie}{Lie}

\DeclareMathOperator{\res}{res}
\DeclareMathOperator{\soc}{soc}

\newcommand{\ve}{\varepsilon}

\newcommand{\Fp}{\F_p}
\newcommand{\Fq}{\F_q}

\newcommand{\Gfp}{G(\Fp)}
\newcommand{\Gfq}{G(\Fq)}

\newcommand{\Bfq}{B(\Fq)}

\newcommand{\Ufq}{U(\Fq)}

\newcommand{\Tfq}{T(\Fq)}

\newcommand{\Ug}{U(\g)}
\newcommand{\ug}{u(\g)}

\numberwithin{equation}{subsection}

\newtheorem{theorem}[equation]{Theorem}
\newtheorem{proposition}[equation]{Proposition}

\newtheorem{corollary}[equation]{Corollary}
\newtheorem{lemma}[equation]{Lemma}

\theoremstyle{definition}

\newtheorem{question}[equation]{Question}
\newtheorem{example}[equation]{Example}
\newtheorem{remark}[equation]{Remark}

\title{Bounding the dimensions of rational cohomology groups}

\author[Bendel]{Christopher P.\ Bendel}
\address{Department of Mathematics \\ University of Wisconsin--Stout \\ Menomonie, WI 54751 }
\email{bendelc@uwstout.edu}

\author[Boe]{Brian D.\ Boe}
\address{Department of Mathematics \\ University of Georgia \\ Athens, GA 30602}
\email{brian@math.uga.edu}

\author[Drupieski]{Christopher M.\ Drupieski}
\address{Department of Mathematical Sciences \\ DePaul University \\ Chicago, IL 60614}
\email{cdrupies@depaul.edu}

\author[Nakano]{Daniel K.\ Nakano}
\address{Department of Mathematics \\ University of Georgia \\ Athens, GA 30602}
\thanks{Research of the fourth author was partially supported by NSF grant  DMS-1002135}
\email{nakano@math.uga.edu}

\author[Parshall]{Brian J.\ Parshall}
\address{Department of Mathematics \\ University of Virginia \\ Charlottesville, VA 22903 }
\thanks{Research of the fifth author was partially supported by NSF grant  DMS-1001900}
\email{bjp8w@virginia.edu}

\author[Pillen]{Cornelius Pillen}
\address{Department of Mathematics \\ University of South Alabama \\ Mobile, AL 36688}
\email{pillen@southalabma.edu}

\author[Wright]{Caroline B.\ Wright}
\address{Louise S. McGehee School \\ 2343 Prytania Street \\New Orleans, LA 70130   }
\email{carriew@mcgeheeschool.com} 


\date{\today}

\begin{document}

\begin{abstract}
Let $k$ be an algebraically closed field of characteristic $p > 0$, and let $G$ be a simple simply-connected algebraic group over $k$ that is defined and split over the prime field $\Fp$. In this paper we investigate situations where the dimension of a rational cohomology group for $G$ can be bounded by a constant times the dimension of the coefficient module. We then demonstrate how our results can be applied to obtain effective bounds on the first cohomology of the symmetric group. We also show how, for finite Chevalley groups, our methods permit significant improvements over previous estimates for the dimensions of second cohomology groups.
\end{abstract}

\maketitle

\section{Introduction}

\subsection{} \label{subsection:AIM2007}

Let $k$ be a field, $S$ be a finite group, and $V$ be an absolutely irreducible $kS$-module on which $S$ acts faithfully. In 1986, Guralnick conjectured the existence of a universal upper bound, independent of $k$, $S$, or $V$, for the dimension of the first cohomology group $\opH^1(S,V)$ \cite{Guralnick:1986}. Based on the evidence available at the time, Guralnick suggested that a suitable upper bound might be $2$, though later work by Scott and his student McDowell showed that if $S = PSL_6(\Fp)$ with $p$ sufficiently large, then there exists an absolutely irreducible $k S$-module $V$ on which $S$ acts faithfully with $\dim \opH^1(S,V) = 3$ \cite{Scott:2003}. Still, the existence of \emph{some} universal upper bound remained a plausible idea, and the best guess for a particular bound remained $3$ until the recent American Institute of Mathematics (AIM) workshop ``Cohomology bounds and growth rates'' in June 2012. As reported by AIM and the workshop organizers \cite{AIM:2012}, on day three of the workshop, Scott reported on calculations conducted by his student Sprowl \cite{Scott:2013}, from which they could deduce the existence of $4$- and $5$-dimensional examples for $\opH^1(S,V)$ when $S = PSL_7(\Fp)$ with $p$ sufficiently large. These calculations were independently confirmed by L\"{u}beck, who subsequently showed that large-dimensional examples also arise when $S = \Gfp$ is a finite group of Lie type with underlying root system of type $E_6$ or $F_4$. Among all the dimensions computed during and in the weeks after the workshop, the largest was $\dim \opH^1(S,V) = 469$ for $S = PSL_8(\Fp)$ with $p$ some sufficiently large prime number. These particular large-dimensional  examples do not disprove Guralnick's conjecture, but they do make it seem less likely that any universal upper bound exists.

Even if no universal upper bound exists for the dimensions of the cohomology groups $\opH^1(S,V)$, the computer calculations of Scott, Sprowl, and L\"{u}beck demonstrate how, by exploiting connections between the cohomology of semisimple algebraic groups and the combinatorics of Kazhdan--Lusztig polynomials, it is possible to obtain much information about the size of $\opH^1(S,V)$ when $S$ is a finite group of Lie type. Indeed, a thread of research leading up to the 2012 AIM workshop, and since, has been the desire to obtain, or show the existence of, bounds on the dimensions of the cohomology groups $\opH^m(\Gfq,V)$ that depend only on the (rank of the) underlying root system. Specifically, let $G$ be a simple simply-connected algebraic group over an algebraically closed field $k$ of characteristic $p > 0$. Assume $G$ is defined and split over $\Fp$. Given $q = p^r$ with $r \geq 1$, let $\Gfq$ be the finite subgroup of $\Fq$-rational points in $G$. Cline, Parshall, and Scott \cite{Cline:2009} proved that there exists a constant $C(\Phi)$, depending only on the underlying root system $\Phi$, such that for each irreducible $k\Gfq$-module $V$,
\begin{equation} \label{eq:CPSbound}
\dim \opH^1(\Gfq,V) \leq C(\Phi).
\end{equation} 
Arguing via different methods, Parker and Stewart \cite{Parker:2012} determined an explicit (large) constant that can be used for $C(\Phi)$ in \eqref{eq:CPSbound}, and which is given by a formula depending on the rank of $\Phi$. On the other hand, if $K$ is an algebraically closed field of characteristic $r$ with $r$ relatively prime to $q$, and if $G'$ is a finite simple group of Lie type, then Guralnick and Tiep showed for each irreducible $KG'$-module $V$ that $\dim \opH^1(G',V) \leq \abs{W} + e$ \cite[Theorem 1.3]{Guralnick:2011}. Here $W$ is the Weyl group of $G'$, and $e$ is the twisted Lie rank of $G'$. Except for certain small values of $q$ depending on the Lie type of $G$, the finite group $\Gfq$ is a central extension of a nonabelian simple group of Lie type; see \cite[2.2.6--2.2.7]{Gorenstein:1998}.

Parshall and Scott \cite{Parshall:2011a} later extended the Cline--Parshall--Scott result \eqref{eq:CPSbound} to show for each irreducible rational $G$-module that 
\begin{equation} \label{eq:PSbound}
\dim\opH^m(G,V)\leq c(\Phi,m)
\end{equation}
for some constant $c(\Phi,m)$ depending only on $\Phi$ and the degree $m$. Then Bendel et al.\ \cite{Bendel:2012} succeeded in finding a simultaneous generalization of \eqref{eq:CPSbound} and \eqref{eq:PSbound}, showing for each irreducible $k\Gfq$-module $V$ that
\begin{equation}
\dim \opH^{m}(\Gfq,V) \leq C(\Phi,m)
\end{equation}
for some constant $C(\Phi,m)$ depending only on $\Phi$ and $m$. Similar results were also obtained in \cite{Parshall:2011a} bounding the higher extension groups $\Ext^m_G(V_1,V_2)$, and in \cite{Bendel:2012} bounding the groups $\Ext^m_{\Gfq}(V_1,V_2)$, assuming that $V_1$ and $V_2$ are irreducible rational $G$-modules (resp.\ $k\Gfq$-modules), though some additional restrictions on $V_1$ are necessary when $m > 1$.

\subsection{}

In a different direction from the results described above, in this paper we explore bounds on the dimensions of cohomology groups that depend not on the rank of an underlying root system, but on the dimension of the coefficient module. This is in the spirit of a number of earlier general results providing bounds on the dimensions of $\opH^1(S,V)$ and $\opH^2(S,V)$ for $S$ a finite group. Specifically, for $m = 1$, Guralnick and Hoffman proved:

\begin{theorem} \textup{\cite[Theorem 1]{Guralnick:1998}} \label{theorem:GH}
Let $S$ be a finite group, and let $V$ be an irreducible $kS$-module on which $S$ acts faithfully. Then 
\[
\dim \opH^1(S,V)\leq \tfrac{1}{2}\dim V.
\]
\end{theorem}

In the case $m=2$, the cohomology group $\opH^2(S,V)$ parametrizes non-equivalent group extensions of $V$ by $S$, and has connections with the study of profinite presentations. Guralnick, Kantor, Kassabov, and Lubotsky verified an earlier conjecture of Holt by proving the following theorem: 

\begin{theorem} \textup{\cite[Theorem B]{Guralnick:2007}} \label{gkkl} 
Let $S$ be a finite quasi-simple group, and let $V$ be a $kS$-module. Then 
\[
\dim \opH^2(S,V)\leq (17.5) \dim V.
\]
\end{theorem}

Guralnick et al.\ also showed that if $S$ is an arbitrary finite group and if $V$ is an irreducible $kS$-module that is faithful for $S$, then $\dim \opH^2(S,V) \leq (18.5) \dim V$ \cite[Theorem C]{Guralnick:2007}, but that in general no analogue of this result can hold for $\opH^m(S,V)$ when $m \geq 3$ \cite[Theorem G]{Guralnick:2007}. Still, their work leaves open the possibility of finding constants $C(m)$ for each $m \geq 3$ such that $\dim \opH^m(S,V) \leq C(m) \cdot \dim V$ when $S$ is restricted to a suitable collection of finite groups. If such constants exist, we say that the cohomology groups $\opH^m(S,V)$, for $S$ in the specified collection of finite groups and $V$ in the specified collection of $kS$-modules, are \emph{linearly bounded}. We call the existence of such constants the \emph{linear boundedness question} for the given groups and modules. More generally, we can consider the linear boundedness question for collections of algebraic groups and for accompanying collections of rational modules.

\subsection{}

This paper investigates the linear boundedness question for the rational cohomology of a simple simply-connected algebraic group $G$ over $k$ that is defined and split over $\Fp$. In other words, we investigate, for $m \leq 3$, upper bounds on the dimension of $\opH^m(G,M)$ for $M$ a rational $G$-module (which the reader may assume to always be finite-dimensional, though we do not always make this assumption explicit, nor is it necessary for the validity of every result in this paper). In this context we are able to exploit the existence of a Borel subgroup $B$ in $G$ (i.e., a maximal closed connected solvable subgroup in $G$) and a maximal torus $T$ in $B$. In Section \ref{section:bounds} we apply intricate calculations of Bendel, Nakano, and Pillen \cite{Bendel:2007}, Wright \cite{Wright:2011}, and Andersen and Rian \cite{Andersen:2007}, summarized in Theorem \ref{theorem:BNPWAR}, to prove for a finite-dimensional rational $B$-module $M$ that
\begin{equation} 
\dim \opH^{m}(B,M) \leq 
\begin{cases} 
\dim M & \text{if $m=1$ or $2$,} \\
2 \dim M   & \text{if $m=3$ and $p>h$.}  
\end{cases} 
\end{equation}
Here $h$ is the Coxeter number for $\Phi$. For a rational $G$-module $M$, it is well known that $\opH^{m}(G,M) \cong \opH^{m}(B,M)$, so the above inequalities yield general bounds on the dimensions of rational cohomology groups for $G$  when $m$ equals $1$, $2$, or $3$. Further refinements are given in the case $m=1$ through the explicit computation of $\opH^{1}(B,\mu)$ for $\mu$ a one-dimensional $B$-module. 

If $M$ is a rational $T$-module, then $M$ admits a weight space decomposition $M = \bigoplus_{\lambda \in X(T)} M_\lambda$. Here $X(T)$ is the character group of $T$. In particular, let $M$ be a rational $G$-module. Using the weight space decomposition of $M$, in Section \ref{section:boundsweightspaces} we establish bounds on the dimension of $ \opH^{1}(G,M)$ in terms of the dimensions of the weight spaces of $M$. This new idea gives much finer information than previous bounds depending only on the dimension of $M$, since it enables us to produce formulas in terms of the differences of dimensions of weight spaces. For example, given knowledge about the weight space decomposition of $M$, one can use these formulas in various situations to prove the vanishing of cohomology groups. More generally, when the dimension of $M$ is relatively small, our bounds provide much more effective estimates on the dimension of $\opH^1(G,M)$ than the estimates that arise through the methods of Parshall and Scott \cite{Parshall:2011a} or of Parker and Stewart \cite{Parker:2012}.

As an application of our techniques, in Section \ref{section:applications} we show how to obtain effective bounds on the dimension of first cohomology groups for $S=\Sigma_{d}$, the symmetric group on $d$ letters. We also demonstrate for $S=\Gfq$, with $q = p^r$ and $r$ sufficiently large, that
\begin{equation}
\dim \opH^{m}(\Gfq,V) \leq 
\begin{cases}
\frac{1}{h} \dim V & \text{if $m=1$,} \\
\dim V & \text{if $m=2$,} \\ 
2\dim V & \text{if $m=3$ and $p>h$}
\end{cases}
\end{equation} 
for each $k\Gfq$-module $V$. Our results for $m=2$ indicate that the bound given in Theorem \ref{gkkl} can be significantly improved when $S$ is a finite Chevalley group. We do not treat the twisted finite groups of Lie type in this paper, but invite the reader to consider how our results could be extended to those cases.

\subsection{Notation} \label{subsection:notation}

We generally follow the notation and terminology of \cite{Jantzen:2003}. Let $k$ be an algebraic\-ally closed field of characteristic $p > 0$. Let $G$ be a simple simply-connected algebraic group scheme over $k$, defined and split over $\Fp$, and let $F: G \rightarrow G$ be the standard Frobenius morphism on $G$. For $r \geq 1$ and $q = p^r$, denote by $G_r$ the $r$-th Frobenius kernel of $G$, and by $\Gfq$ the finite subgroup of $\Fq$-rational points in $G$, consisting of the fixed-points in $G(k)$ of $F^r$. Then $\Gfq$ is the universal version of an untwisted finite group of Lie type, as defined in \cite[2.2]{Gorenstein:1998}.

Let $T \subset G$ be a maximal torus, which we assume to be defined and split over $\Fp$. Let $\Phi$ be the set of roots of $T$ in $G$, and let $h$ be the Coxeter number of $\Phi$. Then $\Phi$ is an indecomposable root system. Fix a set of simple roots $\Delta \subset \Phi$, and denote the corresponding sets of positive and negative roots in $\Phi$ by $\Phi^+$ and $\Phi^-$, respectively. Write $W = N_G(T)/T$ for the Weyl group of $\Phi$. Let $B=T\cdot U \subset G$ be a Borel subgroup containing $T$, with unipotent radical $U$ corresponding to $\Phi^-$. Set $\Bfq = B \cap \Gfq$, $\Ufq = U \cap \Gfq$, and $\Tfq = T \cap \Gfq$. Similarly, set $B_r = B \cap G_r$, $U_r = U \cap G_r$, and $T_r = T \cap G_r$. We write $\g = \Lie(G)$ and $\fu = \Lie(U)$ for the Lie algebras of $G$ and $U$, respectively. Then $\g$ and $\fu$ are naturally restricted Lie algebras. We write $u(\g)$ for the restricted enveloping algebra of $\g$, and $\Ug$ for the ordinary universal enveloping algebra of $\g$.

Let $X(T)$ be the character group of $T$. Write 
\[
X(T)_+ = \set{ \lambda \in X(T): (\lambda,\alpha^\vee) \geq 0 \text{ for all } \alpha \in \Delta}
\]
for the set of dominant weights in $X(T)$, and for $r \geq 1$, write
\[
X_r(T) = \set{\lambda \in X(T)_+: (\lambda,\alpha^\vee) < p^r \text{ for all } \alpha \in \Delta}
\]
for the set of $p^r$-restricted dominant weights in $X(T)$. For each $\lambda\in X(T)_{+}$, let $H^0(\lambda)=\text{ind}_{B}^{G}(\lambda)$ be the corresponding induced module, which has irreducible socle $\soc_G H^0(\lambda) = L(\lambda)$. Each irreducible rational $G$-module is isomorphic to $L(\lambda)$ for some $\lambda \in X(T)_+$. Since $G$ is assumed to be simply-connected, the $L(\lambda)$ for $\lambda \in X_r(T)$ form a complete set of pairwise nonisomorphic irreducible $k\Gfq$-modules.

\subsection{Acknowledgements}

The authors thank the American Institute of Mathematics for hosting the workshops ``Cohomology and representation theory for finite groups of Lie type'' in June 2007, and ``Cohomology bounds and growth rates'' in June 2012. Many of the results described in Section \ref{subsection:AIM2007} were motivated by the ideas exchanged at the 2007 workshop, and provided impetus for the organization of the 2012 meeting. The results of this paper were obtained in the AIM working group format at the 2012 workshop, which promoted a productive exchange of ideas between the authors at the meeting.

The fourth author (Nakano) presented talks at the U.C.\ Lie Theory Workshops at U.C.\ Santa Cruz in 1999 and at Louisiana State University (LSU) in 2011. At the LSU meeting, his lecture was devoted to explaining the various connections between the cohomology theories for reductive algebraic groups and their associated Frobenius kernels and finite Chevalley groups. The results in this paper are a natural extension of the results discussed in his presentation. 

\section{Bounds on rational cohomology groups} \label{section:bounds}

\subsection{Weight spaces and \texorpdfstring{$B$}{B}-cohomology} \label{subsection:Bcohomology}

The irreducible $B$-modules are one-dimensional and are identified with elements of $X(T)$ via inflation from $T$ to $B$. For a finite-dimensional rational $B$-module $M$, the $B$-module composition factors of $M$ can be read off with multiplicities from its weight space decomposition. By considering the long exact sequence in cohomology, it follows for each $m \geq 0$ that one has the inequality
\begin{equation}\label{eq:Bbound} \textstyle
\dim \opH^m(B,M) \leq \sum_{\mu \in X(T)} \dim M_\mu \cdot \dim \opH^m(B,\mu).
\end{equation}
Thus, if one can determine a bound on the dimension of $\opH^m(B,\mu)$ that depends only on $m$ and not on $\mu$, then one can obtain a similar bound on the dimension of $\opH^m(B,M)$ that depends only on $m$ and the dimension of $M$. In particular, if $M$ is a rational $G$-module considered also as a rational $B$-module by restriction, then one has $\Hbul(G,M) \cong \Hbul(B,M)$ by \cite[II.4.7]{Jantzen:2003}, so a bound on the dimension of $\opH^m(B,M)$ automatically yields a bound on the dimension of $\opH^m(G,M)$.

For $m=1$, we can use Andersen's explicit computation of $\opH^{1}(B,\mu)$ for each $\mu \in X(T)$, together with the formula \eqref{eq:Bbound}, to give a general upper bound on $\dim \opH^{1}(B,M)$. 


\begin{theorem}\textup{\cite[Corollary 2.4]{Andersen:1984a}}
Let $\mu \in X(T)$. Then
\[
\opH^1(B,\mu) \cong \begin{cases} k & \text{if }\mu = -p^t \alpha \text{ for some } \alpha \in \Delta \text{ and } t \geq 0, \\ 0 & \text{otherwise.} \end{cases}
\]
\end{theorem}

\begin{corollary} \label{corollary:H1Bbound}
Let $M$ be a finite-dimensional rational $B$-module. Then
\[
\dim \opH^1(B,M) \leq \sum_{\alpha \in \Delta, t \geq 0} \dim M_{-p^t\alpha} \leq \dim M.
\]
In particular, if $M$ is a finite-dimensional rational $G$-module, then
\[
\dim \opH^1(G,M) \leq \sum_{\alpha \in \Delta,\; t \geq 0} \dim M_{-p^t\alpha}.
\]
\end{corollary}

\subsection{Applications to \texorpdfstring{$G$}{G}-cohomology}
\label{subsection:Gcohomology}

We can now apply the results of the preceding section to give bounds on the dimension of $\opH^{1}(G,M)$ for $M$ a rational $G$-module, by considering the action of the Weyl group on the set of weights of $M$.

\begin{theorem} \label{theorem:oneoverh}
Let $M$ be a finite-dimensional rational $G$-module. Then
\[
\dim \opH^1(G,M) \leq \tfrac{1}{h} \dim M.
\]
\end{theorem}

\begin{proof}
Recall that restriction from $G$ to $B$ induces an isomorphism $\Hbul(G,M) \cong \Hbul(B,M)$. Then
\[ \textstyle
\dim \opH^1(G,M) \leq \sum_{\alpha \in \Delta,\; t \geq 0} \dim M_{-p^t\alpha}
\]
by Corollary \ref{corollary:H1Bbound}. Next, the set of weights of a rational $G$-module is invariant under the action of the ambient Weyl group $W$, and all roots in $\Phi$ of a given root length lie in a single $W$-orbit. In particular, if $t \geq 0$, and if $\alpha,\beta \in \Delta$ are of the same length, then $\dim M_{-p^t\alpha} = \dim M_{-p^t \beta}$. Let $\Phi_s$ (resp.\ $\Phi_l$) denote the set of short (resp.\ long) roots in $\Phi$, and set $\Delta_s = \Delta \cap \Phi_s$ (resp.\ $\Delta_l = \Delta \cap \Phi_l$). Then one can check that $\abs{\Phi_s} = h \cdot \abs{\Delta_s}$ and $\abs{\Phi_l} = h \cdot \abs{\Delta_l}$; cf.\ \cite[Proposition 3.18]{Humphreys:1990} for the case of one root length. Together these equalities imply that $\sum_{\alpha \in \Delta,\; t \geq 0} \dim M_{-p^t\alpha} \leq \frac{1}{h} \dim M$.
\end{proof}

\begin{corollary} \label{corollary:H1halfbound}
Let $M$ be a finite-dimensional rational $G$-module. Then
\[
\dim \opH^1(G,M) \leq \tfrac{1}{2} \dim M.
\]
\end{corollary}

\begin{proof}
This follows from Theorem \ref{theorem:oneoverh} since the Coxeter number is always at least $2$.
\end{proof}

If $M = L(\lambda)$ is an irreducible rational $G$-module, then we can use Theorem \ref{theorem:GH} to give an alternate proof of Corollary \ref{corollary:H1halfbound}, as follows. To begin, we may assume by the Linkage Principle that $\lambda \in \Z\Phi$, and also that $\lambda \neq 0$, since $L(0) = k$ and $\opH^1(G,k) = 0$ \cite[II.4.11]{Jantzen:2003}. Next, choose $r > 1$ such that $\lambda \in X_r(T)$, and set $q = p^r$. Then $L(\lambda)$ is an irreducible $k\Gfq$-module, and the restriction map $\opH^1(G,L(\lambda)) \rightarrow \opH^1(\Gfq,L(\lambda))$ is injective by \cite[7.4]{Cline:1977a}. Now write $Z(\Gfq)$ for the center of $\Gfq$. Since $G$ is simply-connected (i.e., is of universal type), $G'(\Fq) :=\Gfq/Z(\Gfq)$ is a nonabelian finite simple group by \cite[2.2.6--2.2.7]{Gorenstein:1998}; this uses the fact that $r> 1$. Also, $Z(\Gfq)$ is a subgroup of $Z(G)$ by \cite[2.5.9]{Gorenstein:1998}. Since $Z(G) = \bigcap_{\alpha \in \Phi} \ker(\alpha) \subset T$ acts trivially on $L(\lambda)$ whenever $\lambda \in \Z\Phi$, it follows that $L(\lambda)$ is naturally a nontrivial irreducible module for $G'(\Fq)$. In particular, $G'(\Fq)$ must act faithfully on $L(\lambda)$. The group $Z(\Gfq)$ has order prime to $p$, so it follows from considering the Lyndon--Hochschild--Serre spectral sequence for the group extension $1 \rightarrow Z(\Gfq) \rightarrow \Gfq \rightarrow G'(\Fq) \rightarrow 1$ that $\opH^1(\Gfq,L(\lambda)) \cong \opH^1(G'(\Fq),L(\lambda))$. Then $\dim \opH^1(\Gfq,L(\lambda)) \leq \frac{1}{2} \dim L(\lambda)$ by \cite[Theorem 1]{Guralnick:1998}.

\subsection{Bounds for second and third cohomology groups} \label{subsection:boundssecondthird}

Now we consider $\opH^{m}(B,\mu)$, $\opH^m(B,M)$, and $\opH^m(G,M)$ for $m$ equal to $2$ or $3$. First recall the following results:

\begin{theorem} \label{theorem:BNPWAR}
Let $\mu \in X(T)$. Then
\begin{itemize} 
\item[\textup{(a)}] $\dim \opH^2(B,\mu) \leq 1$.
\item[\textup{(b)}] If $p > h$, then $\dim \opH^3(B,\mu) \leq 2$.
\end{itemize} 
\end{theorem}

\begin{proof}
For (a), see \cite[Theorem 5.8]{Bendel:2007} and \cite[Theorem 4.1.1]{Wright:2011}. For (b), see \cite[Theorem 5.2]{Andersen:2007}.
\end{proof}

Applying (\ref{eq:Bbound}) to the preceding theorem, and using the fact that $\opH^{\bullet}(G,M) \cong \opH^{\bullet}(B,M)$ for each rational $G$-module $M$, one obtains:

\begin{corollary} \label{corollary:boundsonBcoho}
Let $M$ be a finite-dimensional rational $B$-module. Then 
\begin{itemize}
\item[\textup{(a)}]  $\dim \opH^2(B,M) \leq \dim M$.
\item[\textup{(b)}] If $p > h$, then $\dim \opH^3(B,M) \leq 2 \cdot \dim M$.
\end{itemize}
In particular, if $M$ is a finite-dimensional rational $G$-module, then these inequalities also hold with $B$ replaced by $G$.
\end{corollary}

As in Corollary \ref{corollary:H1Bbound}, a stronger version of the above result can be obtained by considering precisely which weights $\mu$ of $M$ satisfy $\opH^i(B,\mu) \neq 0$. The results in the corollary motivate posing the following question, an affirmative answer to which would yield, for each $m \geq 1$, upper bounds on the dimensions of $\opH^m(B,M)$ and $\opH^m(G,M)$ for each finite-dimensional rational $G$-module $M$.

\begin{question}[Linear boundedness for Borel subgroups]
For each $m \geq 1$, does there exist a constant $C(m)$, depending on $m$ but independent of the rank of $G$ or of the weight $\mu \in X(T)$, such that $\dim \opH^m(B,\mu) \leq C(m)$?
\end{question}

\subsection{Bounds for finite Chevalley groups}

In this section we discuss some rough analogues for the finite subgroup $\Gfq$ of $G$ of the results in Sections \ref{subsection:Gcohomology} and \ref{subsection:boundssecondthird}. While the bounds presented here are often significantly worse than those given by Theorem \ref{theorem:GH}, we point out that they can be obtained using only purely elementary methods. Recall that $p$ is \emph{nonsingular} for $G$ if $p > 2$ when $\Phi$ is of type $B$, $C$, or $F$, and if $p > 3$ when $\Phi$ is of type $G_2$.

\begin{theorem} \label{theorem:ranktimesdimension}
Let $M$ be a finite-dimensional $k\Gfq$-module, and suppose $p$ is nonsingular for $G$. Then
\[
\dim \opH^1(\Gfq,M) \leq r \cdot \abs{\Delta} \cdot \dim M.
\]
\end{theorem}

\begin{proof}
Recall that $\Ufq$ is a Sylow p-subgroup of $\Gfq$ \cite[2.3.4]{Gorenstein:1998}. In particular, the index of $\Ufq$ in $\Gfq$ is prime to $p$, so restriction from $\Gfq$ to $\Ufq$ defines an injection
\[
\Hbul(\Gfq,M) \hookrightarrow \Hbul(\Ufq,M).
\]
Since $\Ufq$ is a $p$-group, each irreducible $k\Ufq$-module is isomorphic to $k$. Then considering a $\Ufq$-composition series for $M$, and using the long exact sequence in cohomology, it follows by induction on the dimension of $M$ that
\[
\dim \opH^1(\Ufq,M) \leq \dim \opH^1(\Ufq,k) \cdot \dim M.
\]
Now $\opH^1(\Ufq,k)$ identifies with the space of $k$-linear maps $k \Ufq_{ab} \rightarrow k$. Here $\Ufq_{ab}$ is the abelianization of $\Ufq$. Since $p$ is nonsingular, it follows from \cite[3.3.1]{Gorenstein:1998} that $\Ufq_{ab} \cong (\Fq)^{\abs{\Delta}}$ as an abelian group. Specifically, $\Ufq_{ab}$ identifies with the direct product of the root subgroups in $\Ufq$ corresponding to simple roots. Since $q = p^r$, $k \Fq = k \otimes_{\Fp} \Fq$ has $k$-dimension $r$. Then $\dim \opH^1(\Ufq,k) = r \cdot \abs{\Delta}$.
\end{proof}

More generally, one can argue as in the proof of the theorem to show for each $m \geq 0$ that
\[
\dim \opH^m(\Gfq,M) \leq \dim \opH^m(\Ufq,k) \cdot \dim M.
\]
In turn, the dimension of $\opH^m(\Ufq,k)$ is bounded above by the dimension of the cohomology group $\opH^m(u(\fu^{\oplus r}),k)$ for the restricted enveloping algebra $u(\fu^{\oplus r})$; see \cite[2.4]{UGA:2013}.

\section{Bounds depending on weight space multiplicities} \label{section:boundsweightspaces}

\subsection{}

Next we explore some bounds on $\dim \opH^1(G,M)$ that depend on weight space multiplicities.

\begin{lemma}
Let $\lambda \in X(T)_+$. Then $\dim \opH^1(G,L(\lambda)) \leq \dim H^0(\lambda)_0$.
\end{lemma}

\begin{proof} If $\lambda=0$, then the result follows because $H^0(0) = L(0) = k$ and  $\opH^1(G,k)=0$. So assume that $\lambda \neq 0$.
There exists a short exact sequence of $G$-modules
\[
0 \rightarrow L(\lambda) \rightarrow H^0(\lambda) \rightarrow Q \rightarrow 0.
\]
Since $\opH^1(G,H^0(\lambda)) = 0$ by \cite[II.4.13]{Jantzen:2003}, and since $\Hom_G(k,L(\lambda)) = 0$ by the assumption $\lambda \neq 0$, the corresponding long exact sequence in cohomology yields $\opH^1(G,L(\lambda)) \cong \Hom_G(k,Q)$. Now the multiplicity of the trivial module in $\soc_G Q$ is bounded above by the composition multiplicity of the trivial module in $H^0(\lambda)$, which is bounded above by the weight space multiplicity $\dim H^0(\lambda)_0$.
\end{proof}

In some cases we can improve the conclusion of the lemma to a strict inequality. Suppose that the set of weights of $T$ in $L(\lambda)$ is equal to that of $H^0(\lambda)$. By \cite{Premet:1987}, this condition is known to hold if $\lambda \in X_1(T)$ and $p$ is good for $\Phi$. Recall that $p$ is good for $\Phi$ provided $p > 2$ when $\Phi$ is of type $B_n$, $C_n$, or $D_n$; $p > 3$ when $\Phi$ is of type $F_4$, $G_2$, $E_6$, or $E_7$; and provided $p > 5$ when $\Phi$ is of type $E_8$. By the Linkage Principle, $\opH^1(G,L(\lambda)) = 0$ unless $\lambda \in \Z\Phi$. It is well known that the set of weights of $H^0(\lambda)$ is saturated. Thus, if $\lambda \in \Z\Phi \cap X(T)_+$, then $0$ is a weight of $H^0(\lambda)$, and hence also of $L(\lambda)$. In particular, $\dim Q_0 < \dim H^0(\lambda)_0$. Since $\dim \Hom_G(k,Q)$ is bounded above by $\dim Q_0$, we obtain in this case that $\dim \opH^1(G,L(\lambda)) < \dim H^0(\lambda)_0$.

\subsection{} \label{subsection:relatetoLiealgebra}

We now present a preliminary result that relates $\opH^{1}(G,M)$ to cohomology for $\g = \Lie(G)$, by way of cohomology for the Frobenius kernel $G_1$. Write $\Hbul(\g,M) = \Hbul(\Ug,M)$ for the ordinary Lie algebra cohomology of $\g$ with coefficients in the $\g$-module $M$. If $M$ is a rational $G$-module, then the adjoint action of $G$ on $\g$, together with the given action of $G$ on $M$, induce a rational $G$-module structure on $\Hbul(\g,M)$.

\begin{lemma} \label{lemma:Exttransfer}
Let $M$ be a rational $G$-module, and suppose that $M^{G_1} = 0$. Then
\begin{itemize}
\item[\textup{(a)}] Restriction from $G$ to $G_1$ induces an isomorphism $\opH^1(G,M) \cong \opH^1(G_1,M)^{G/G_1}$.
\item[\textup{(b)}] There exists a $G$-equivariant isomorphism $\opH^1(G_1,M) \cong \opH^1(\g,M)$.
\item[\textup{(c)}] $\opH^1(G,M) \cong \opH^1(\g,M)^G$.
\end{itemize}
\end{lemma}

\begin{proof}
Consider the Lyndon--Hochschild--Serre spectral sequence
\begin{equation} \label{eq:LHSforG1}
E_2^{i,j} = \opH^i(G/G_1,\opH^j(G_1,M)) \Rightarrow \opH^{i+j}(G,M),
\end{equation}
and its associated five-term exact sequence
\begin{equation} \label{eq:LHSfiveterm}
0 \rightarrow E_2^{1,0} \rightarrow \opH^1(G,M) \rightarrow E_2^{0,1} \rightarrow E_2^{2,0} \rightarrow \opH^2(G,M).
\end{equation}
Since $\Hom_{G_1}(k,M) = M^{G_1} = 0$, one has $E_2^{i,0} = 0$ for all $i \geq 0$. Then \eqref{eq:LHSfiveterm} yields that restriction from $G$ to $G_1$ defines an isomorphism $\opH^1(G,M) \cong \opH^1(G_1,M)^{G/G_1}$. This proves (a).

For (b), recall that the representation theory of the restricted enveloping algebra $\ug$ is naturally equivalent to that of the Frobenius kernel $G_1$ \cite[I.9.6]{Jantzen:2003}. We thus identify $\opH^1(G_1,M)$ and $\opH^1(\ug,M)$ via this equivalence. Since $\ug$ is a homomorphic image of $\Ug$, and since the quotient map $\Ug \rightarrow \ug$ is compatible with the adjoint action of $G$, there exists a corresponding $G$-module homomorphism $\opH^1(G_1,M) \rightarrow \opH^1(\g,M)$, which by \cite[I.9.19(1)]{Jantzen:2003} fits into an exact sequence
\begin{equation} \label{eq:ordinarycohomologyexactseq}
0 \rightarrow \opH^1(G_1,M) \rightarrow \opH^1(\g,M) \rightarrow \Hom^s(\g,M^\g).
\end{equation}
Here, given a vector space $V$, $\Hom^s(\g,V)$ denotes the set of additive functions $\varphi:\g \rightarrow V$ satisfying the property $\varphi(ax) = a^p\varphi(x)$ for all $a \in k$ and $x \in \g$. Since $M$ is a rational $G$-module, it is in particular a restricted $\g$-module, i.e., a $\ug$-module. Then $M^{\g} = M^{\Ug} = M^{\ug} = M^{G_1}$. This space is zero by assumption, so we conclude that $\opH^1(G_1,M) \cong \opH^1(\g,M)$. This proves (b). Now (c) follows immediately from (a) and (b).  
\end{proof}

Replacing $G$ by $B$ in the previous proof, one obtains the following lemma:

\begin{lemma} \label{lemma:BtoU1}
Let $M$ be a rational $B$-module, and suppose that $M^{B_1} = 0$. Then restriction from $B$ to $B_1$ induces an isomorphism $\opH^1(B,M) \cong \opH^1(B_1,M)^{B/B_1}$. In particular, restriction from $B$ to $U_1$ defines an injection $\opH^1(B,M) \hookrightarrow \opH^1(U_1,M)^T$.
\end{lemma}

\begin{proof}
It remains to explain the last statement in the lemma. Since $B_1$ is the semidirect product of $U_1$ and the diagonalizable group scheme $T_1$, it follows that restriction from $B_1$ to $U_1$ defines an isomorphism $\Hbul(B_1,M) \cong \Hbul(U_1,M)^{T_1}$ \cite[I.6.9]{Jantzen:2003}. Then restriction from $B$ to $U_1$ defines an isomorphism $\opH^1(B,M) \cong \opH^1(U_1,M)^{B/U_1}$, and the latter space is a subspace of $\opH^1(U_1,M)^T$.
\end{proof}

\subsection{} \label{subsection:sumsimpleroots}

The results in the preceding section can be employed to establish upper bounds for $\opH^{1}(G,M)$ in terms of specific weight space multiplicities.

\begin{proposition} \label{proposition:sumsimpleroots}
Let $M$ be a rational $B$-module, and suppose that $M^{B_1} = 0$. Let $\Delta' \subseteq \Phi^+$ be a set of roots such that the root spaces $\fu_{-\alpha}$ for $\alpha \in \Delta'$ generate $\fu$ as a Lie algebra. Then
\[ \textstyle
\dim \opH^1(B,M) \leq \sum_{\alpha \in \Delta'} \dim M_{-\alpha} - \dim M_0.
\]
\end{proposition}

\begin{proof}
First, $\opH^1(B,M)$ injects into $\opH^1(U_1,M)^T$ by Lemma \ref{lemma:BtoU1}. Next, replacing $G$ by $U$ in \eqref{eq:ordinarycohomologyexactseq}, there exists a $B$-equivariant injection $\opH^1(U_1,M) \hookrightarrow \opH^1(\fu,M)$. Now recall that $\opH^1(\fu,M)$ fits into an exact sequence
\begin{equation} \label{eq:DerSES}
0 \rightarrow \Inn(\fu,M) \rightarrow \Der(\fu,M) \rightarrow \opH^1(\fu,M) \rightarrow 0.
\end{equation}
Here $\Der(\fu,M)$ is the space of all Lie algebra derivations of $\fu$ into $M$, and $\Inn(\fu,M)$ is the space of all inner derivations of $\fu$ into $M$. Since $M$ is a rational $B$-module, the conjugation action of $B$ on $\fu$ makes \eqref{eq:DerSES} into an exact sequence of rational $B$-modules. Then applying the exact functor $(-)^T$ to \eqref{eq:DerSES}, one obtains
\[
\dim \opH^1(B,M) \leq \dim \opH^1(\fu,M)^T = \dim \Der(\fu,M)^T - \dim \Inn(\fu,M)^T.
\]
As rational $B$-modules, $\Inn(\fu,M) \cong M/M^{\fu}$. Observe that $(M^{\fu})^T = (M^{U_1})^T \subseteq M^{B_1} = 0$. Then it follows that $\Inn(\fu,M)^T \cong (M/M^\fu)^T \cong M_0$. Finally, a Lie algebra derivation $\fu \rightarrow M$ is completely determined by its action on a set of Lie algebra generators for $\fu$, say, the root subspaces $\fu_{-\alpha}$ for $\alpha \in \Delta'$. Moreover, a $T$-invariant derivation $\fu \rightarrow M$ must map $\fu_{-\beta}$ into $M_{-\beta}$ for each $\beta \in \Phi^+$. Then $\dim \Der(\fu,M)^T \leq \sum_{\alpha \in \Delta'} \dim M_{-\alpha}$. Combining this with the previous observations, we obtain the inequality in the statement of the proposition.
\end{proof}

\begin{remark}
If $p$ is nonsingular for $G$, then one can take $\Delta' = \Delta$ in Proposition \ref{proposition:sumsimpleroots}.
\end{remark}

The next result is an analogue for algebraic groups of \cite[Corollary 2.9]{Cline:1975}. Let $\alpha_0$ be the highest short root in $\Phi$, and let $\wtalpha$ be the highest long root in $\Phi$.

\begin{corollary} \label{corollary:sumsimpleroots}
Let $M$ be a rational $G$-module, and suppose that $M^{B_1} = 0$. Let $\Delta' \subseteq \Phi^+$ be a set of roots such that the root spaces $\fu_{-\alpha}$ for $\alpha \in \Delta'$ generate $\fu$ as a Lie algebra. Then
\[ \textstyle
\dim \opH^1(G,M) \leq \sum_{\alpha \in \Delta'} \dim M_{\alpha} - \dim M_0.
\]
In particular, suppose $p$ is nonsingular for $G$. Then
\[
\dim \opH^1(G,M) \leq \abs{\Delta_s} \cdot \dim M_{\alpha_0} + \abs{\Delta_l} \cdot \dim M_{\wtalpha} - \dim M_0,
\]
where by convention we consider all roots in $\Phi$ as long, and set $\Delta_s = \emptyset$, whenever $\Phi$ has roots of only a single root length.
\end{corollary}

\begin{proof}
Observe that if the root spaces $\fu_{-\alpha}$ for $\alpha \in \Delta'$ generate $\fu$ as a Lie algebra, then so do the root spaces $\fu_{w_0\alpha}$ for $\alpha \in \Delta'$. Here $w_0 \in W$ is the longest element of $W$. One has $\Hbul(G,M) \cong \Hbul(B,M)$, so $\dim \opH^1(G,M) \leq \sum_{\alpha \in \Delta'} \dim M_{w_0\alpha} - \dim M_0$. But $\dim M_\lambda = \dim M_{w\lambda}$ whenever $w \in W$, so we conclude that $\dim \opH^1(G,M) \leq \sum_{\alpha \in \Delta'} \dim M_{\alpha} - \dim M_0$. In particular, if $p$ is nonsingular for $G$, then we can take $\Delta' = \Delta$. Since all roots of a given root length in $\Phi$ are conjugate under $W$, we then have $\dim M_\alpha = \dim M_{\alpha_0}$ whenever $\alpha \in \Delta_s$, and $\dim M_\alpha = \dim M_{\wtalpha}$ whenever $\alpha \in \Delta_l$, so that $\sum_{\alpha \in \Delta} \dim M_\alpha = \abs{\Delta_s} \cdot \dim M_{\alpha_0} + \abs{\Delta_l} \cdot \dim M_{\wtalpha}$.
\end{proof}

The bounds established in Corollary \ref{corollary:sumsimpleroots} can be improved if we assume that the Weyl group has order prime to $p$. For the next theorem, recall that if $M$ is a rational $G$-module, then the $0$-weight space $M_0$ of $M$ is naturally a module for $W = N_G(T)/T$.

\begin{theorem} \label{theorem:highestrootdimension}
Let $M$ be a rational $G$-module. Assume that $M^{G_1} = 0$, and that $p \nmid \abs{W}$. Then
\[
\dim \opH^1(G,M) \leq \dim M_{\alpha_0} + \dim M_{\wtalpha} - \dim M_0,
\]
where by convention we say $M_{\alpha_0} = 0$ if $\Phi$ has roots of only a single root length. In particular,
\[
\dim \opH^1(G,M) \leq \begin{cases}
\frac{1}{\abs{\Phi}} \cdot \dim M & \text{if $\Phi$ has one root length,} \\
\frac{2}{\abs{\Phi}} \cdot \dim M & \text{if $\Phi$ has two root lengths.} \end{cases}
\]
\end{theorem}

\begin{proof}
The second statement follows from the first by applying the argument given in the proof of Theorem~\ref{theorem:oneoverh}, so we proceed to prove the first statement. Set $N = N_G(T)$, and observe that the fixed-point functor $(-)^N$ factors as the composition of the exact functor $(-)^T$ with the functor $(-)^{N/T} = (-)^W$, which is also exact by the assumption that the finite group $W$ has order prime to $p$. Then $(-)^N$ is exact. One has $\opH^1(G,M) \cong \opH^1(\g,M)^G$ by Lemma~\ref{lemma:Exttransfer}(c), so in particular $\dim \opH^1(G,M) \leq \dim \opH^1(\g,M)^N$. Now applying the exact functor $(-)^N$ to the exact sequence of rational $G$-modules
\[
0 \rightarrow \Inn(\g,M) \rightarrow \Der(\g,M) \rightarrow \opH^1(\g,M) \rightarrow 0,
\]
one obtains
\[
\dim \opH^1(\g,M)^N = \dim \Der(\g,M)^N - \dim \Inn(\g,M)^N.
\]
As a rational $G$-module, $\Inn(\g,M) \cong M/M^{\g}$. But $M^{\g} = M^{G_1} = 0$ by assumption, so
\[
\dim \opH^1(\g,M)^N = \dim \Der(\g,M)^N - \dim M^N.
\]

Observe that $M^N = (M^T)^W = M_0^W$. Also, an $N$-invariant derivation $\delta: \g \rightarrow M$ is in particular $T$-invariant, and so must map $\g_\beta$ into $M_\beta$ for each $\beta \in \Phi$. All roots of a given length in $\Phi$ are conjugate under $W$, so it follows that an $N$-invariant derivation $\delta$ is uniquely determined by its values on any single root space in $\g$ if $\Phi$ has only one root length, and by its values on any pair of root spaces in $\g$ corresponding to a long root and a short root if $\Phi$ has two root lengths. In particular, $\dim \Der(\g,M)^N \leq \dim M_{\alpha_0} + \dim M_{\wtalpha}$, where by convention we say that $M_{\alpha_0} = 0$ if $\Phi$ only has roots of a single root length. Combining this and the preceding observations, one obtains the first statement of the theorem.
\end{proof}


\begin{remark} \label{remark:ptimesdominant}
The assumption $M^{B_1} = 0$, and hence also $M^{G_1} = 0$, is satisfied if $M = L(\lambda)$ for some $\lambda \in X(T)_+$ with $\lambda \notin pX(T)$. Indeed, in this case $\lambda = \lambda_0 + p \lambda_1$ for some $0 \neq \lambda_0 \in X_1(T)$ and some $\lambda_1 \in X(T)_+$. Then $L(\lambda) \cong L(\lambda_0) \otimes L(\lambda_1)^{(1)}$ by the Steinberg tensor product theorem, so that $L(\lambda)^{B_1} \cong L(\lambda_0)^{B_1} \otimes L(\lambda_1)^{(1)}$. Now $L(\lambda_0)^{U_1} = L(\lambda_0)_{w_0\lambda_0}$ (cf.\ \cite[II.3.12]{Jantzen:2003}), so it follows that $L(\lambda_0)^{B_1} = (L(\lambda_0)^{U_1})^{T_1} = 0$ because $w_0\lambda_0 \notin pX(T)$ by the condition $0 \neq \lambda_0 \in X_1(T)$.

If $\lambda = 0$, then $L(\lambda) = k$, and one has $\opH^1(G,k) = 0$ and $\opH^1(\g,k) \cong (\g/[\g,\g])^* = 0$, so that Lemma \ref{lemma:Exttransfer}(c) holds in this case. If also $p \neq 2$ when $\Phi$ is of type $C_n$, then $\opH^1(G_1,k) = 0$ \cite[II.12.2]{Jantzen:2003}, which recovers all parts of Lemma \ref{lemma:Exttransfer} when $M = k$. Now let $\lambda \in X(T)_+ \cap pX(T)$ be nonzero. Then $\lambda = p^s\mu$ for some $\mu \in X(T)_+$ with $\mu \notin p X(T)$, and $L(\lambda) \cong L(\mu)^{(s)}$ is trivial as a $G_s$-module. Suppose that $p \neq 2$ if $\Phi$ is of type $C_n$. Then $\opH^1(G_s,L(\lambda)) \cong \opH^1(G_s,k) \otimes L(\lambda) = 0$ by \cite[II.12.2]{Jantzen:2003}, and $\Hom_{G_s}(k,L(\lambda)) \cong L(\lambda)$, so replacing $G_1$ by $G_s$ in \eqref{eq:LHSforG1}, the corresponding five-term exact sequence shows that the inflation map induces an isomorphism
\[
\opH^1(G/G_s,L(\mu)^{(s)}) \cong \opH^1(G,L(\lambda)).
\]
Identifying $G/G_s$ with the Frobenius twist $G^{(s)}$ of $G$, we can make the identification
\[
\opH^1(G/G_s,L(\mu)^{(s)}) \cong \opH^1(G,L(\mu)).
\]
Then there exists a vector space isomorphism $\opH^1(G,L(\lambda)) \cong \opH^1(G,L(\mu))$. This shows that, assuming that $p \neq 2$ when $\Phi$ is of type $C_n$, Corollary \ref{corollary:sumsimpleroots} and Theorem \ref{theorem:highestrootdimension} can still be applied to provide bounds on $\dim \opH^1(G,L(\lambda))$ when $\lambda \in pX(T)$.
\end{remark}

\begin{remark} \label{remark:reductive}
The results in Sections \ref{subsection:relatetoLiealgebra} and \ref{subsection:sumsimpleroots} remain true, with exactly the same proofs, under the weaker assumption that $G$ is a connected reductive algebraic group over $k$ that is defined and split over $\Fp$, that $T \subset G$ is a maximal split torus in $G$, that the root system $\Phi$ of $T$ in $G$ is indecomposable, that $B \subset G$ is a Borel subgroup of $G$, etc. In particular, the results hold for $G = GL_n(k)$ when $n \geq 2$.
\end{remark}

\section{Applications} \label{section:applications}

\subsection{Bounds for \texorpdfstring{$\Sigma_{d}$}{Sd}}

In this section only, let $S=\Sigma_{d}$ be the symmetric group on $d$ letters ($d \geq 2$), and let $G=GL_d(k)$ be the general linear group. It is well known, from considering commuting actions on tensor space, that there are close connections between the representation theories of $S$ and $G$.

Let $T \subset G$ be the subgroup of diagonal matrices, and write $X(T) = \bigoplus_{i=1}^d \Z \ve_i$ for the character group of $T$. Here $\ve_i: T \rightarrow k$ is the $i$-th diagonal coordinate function on $T$. Recall that the set $X(T)_+$ of dominant weights on $T$ consists of the weights $\lambda = \sum_{i=1}^d a_i\ve_i \in X(T)$ with $a_i-a_{i+1} \geq 0$ for each $1 \leq i < d$. Then identifying a partition $\lambda = (\lambda_1,\lambda_2,\ldots)$ of $d$ with the weight $\lambda = \sum_{i=1}^d \lambda_i \ve_i$, the set of partitions of $d$ is naturally a subset of $X(T)_+$. Moreover, this subset parametrizes the irreducible degree-$d$ polynomial representations of $G$. Recall that a partition $\lambda = (\lambda_1,\lambda_2,\ldots)$ is $p$-restricted if $\lambda_i - \lambda_{i+1} < p$ for each $i \geq 0$, and is $p$-regular if no nonzero part $\lambda_i$ of the partition is repeated $p$ or more times. Then the irreducible $k\Sigma_d$-modules are indexed by the set $\Lambda_{\res}$ of $p$-restricted partitions of $d$. Given $\lambda \in \Lambda_{\res}$, let $D_{\lambda}$ be the corresponding irreducible $k\Sigma_d$-module, and write $\text{sgn}$ for the sign representation of $\Sigma_d$. Given a partition $\lambda$, write $\lambda^{\prime}$ for the transpose partition. Then the irreducible $k\Sigma_d$-modules can be indexed by $p$-regular partitions by setting $D^{\lambda^{\prime}}=D_{\lambda}\otimes \text{sgn}$ for each $\lambda \in \Lambda_{\res}$.

Doty, Erdmann, and Nakano constructed a spectral sequence relating the cohomology theories for $GL_d$ and $\Sigma_d$, showing for $p \geq 3$ and $\lambda \in \Lambda_{\res}$ that
\begin{equation} \label{equation:H1-symmetric}
\opH^{1}(\Sigma_{d},D^{\lambda^{\prime}}) = \opH^{1}(\Sigma_{d},D_{\lambda}\otimes \text{sgn})\cong \Ext^{1}_{G}(\delta,L(\lambda)),
\end{equation} 
where $\delta = (1^d)$ is the one-dimensional determinant representation of $GL_d$ \cite[Theorem 5.4(a)]{Doty:2004}; cf.\ also \cite[Theorem 4.6(b)]{Parshall:2005}. One can now apply Corollary \ref{corollary:sumsimpleroots} to obtain the following bound for the first cohomology of symmetric groups. 

\begin{theorem} \label{theorem:symmetricgroupbound}
Suppose $p\geq 3$, and let $\lambda \in \Lambda_{\res}$ with $\lambda\neq (1^{d})$. Then 
\[ \textstyle
\dim \opH^{1}(\Sigma_{d},D^{\lambda^{\prime}})\leq \sum_{\alpha\in \Delta} \dim L(\lambda)_{\delta+\alpha}-\dim L(\lambda)_{\delta}.
\]
\end{theorem} 

\begin{proof}
One has $\opH^1(\Sigma_d,D^{\lambda'}) \cong \Ext_G^1(\delta,L(\lambda))$ by \eqref{equation:H1-symmetric}. Then by Corollary \ref{corollary:sumsimpleroots} and Remark \ref{remark:reductive},
\begin{align*} 
\dim \Ext^{1}_{G}(\delta,L(\lambda)) &= \dim \opH^{1}(G,L(\lambda)\otimes  -\delta)\\
&\leq  \textstyle \sum_{\alpha\in \Delta} \dim (L(\lambda) \otimes -\delta)_{\alpha}-\dim (L(\lambda)\otimes -\delta)_{0} \\
&= \textstyle \sum_{\alpha\in \Delta} \dim L(\lambda)_{\delta+\alpha}-\dim L(\lambda)_{\delta}. \qedhere
\end{align*}
\end{proof} 

We present the following example to illustrate how Theorem \ref{theorem:symmetricgroupbound} can provide a more effective upper bound on cohomology than earlier established bounds involving $\dim D^{\lambda'}$. 

\begin{example}
Let $\chr(k)=p\geq 3$, and consider $S=\Sigma_n$ with $p \mid n$. In this case
\[
\Delta =\set{\alpha_1,\alpha_2,\ldots,\alpha_{n-1} } = \set{ \epsilon_1-\epsilon_2,\epsilon_2-\epsilon_3,\ldots, \epsilon_{n-1}-\epsilon_{n} },
\]
and $\delta=(1,1,\dots,1)$. Set $\lambda=(2,1,1,\dots,1,0)$. Then $\lambda^{\prime}=(n-1,1,0,\dots,0)$. Moreover, $L(\lambda)$ identifies with the $(n^2-2)$-dimensional irreducible $G=GL_n(k)$-module that can be realized as a quotient of the adjoint representation of $G$ tensored by $\delta$. Also, observe that $\dim D_{\lambda}=\dim D^{\lambda^{\prime}}=\dim L(2,1,\dots, 1,0)_{\delta}=n-2$. Now by Theorem \ref{theorem:symmetricgroupbound},
\[ \textstyle
\dim \opH^1(\Sigma_n,D^{(n-1.1,0,\dots,0)}) \leq \sum_{\alpha \in \Delta} \dim L(\lambda)_{\delta+\alpha} - \dim L(\lambda)_\delta = \abs{\Delta}-(n-2) = 1.
\]
This bound is an equality, because the $(n-1)$-dimensional Specht module $S^{\lambda'}$ is a nonsplit extension of $D^{\lambda'}$ by the trivial module $k$ \cite[Theorem 24.1]{James:1978}. The equality $\dim \Ext_G^1(\delta,L(\lambda)) = 1$ can also be seen from observing that the Weyl module $\Delta(\lambda)$ for $G$ is a nonsplit extension of $L(\lambda)$ by $\delta$. On the other hand, for $p \geq 5$, we claim that Theorem \ref{theorem:GH} yields the (weaker) estimate
\[ \textstyle
\dim \opH^1(\Sigma_n,D^{(n-1,1,0,\dots,0)}) \leq \frac{1}{2} \dim D^{(n-1,1,0,\dots,0)} =\frac{1}{2}(n-2).
\]
In order to apply Theorem \ref{theorem:GH}, we must explain for $p \geq 5$ why the action of $\Sigma_n$ on $D^{\lambda'}$ is faithful. Write $\rho: \Sigma_n \rightarrow GL(D^{\lambda'})$ for the map defining the representation of $\Sigma_n$ on $D^{\lambda'}$, and write $A_n$ for the alternating group on $n$ letters. Observe that $\ker(\rho) \cap A_n$ is a normal subgroup of the nonabelian simple group $A_n$, so either $\ker(\rho) \cap A_n = \set{1}$, or $\ker(\rho) \cap A_n = A_n$. The latter equality is false, because $A_n \not\subset \ker(\rho)$, so we have $\ker(\rho) \cap A_n = \set{1}$. This implies that $\ker(\rho)$ contains only odd permutations and the identity. Since the product of any two odd permutations is an element of $A_n$, and since $A_n \cap \ker(\rho) = \set{1}$, we conclude that in fact $\ker(\rho)$ can contain at most two elements, namely, the identity and an odd permutation that is equal to its own inverse. Subgroups of this type are not normal in $\Sigma_n$, whereas $\ker(\rho)$ is normal in $\Sigma_n$, so we conclude that $\ker(\rho) = \set{1}$. Thus, $D^{\lambda'}$ is an irreducible faithful $\Sigma_n$-module.
\end{example}

\subsection{Bounds for \texorpdfstring{$\Gfq$}{GFq}.}

Assume once again that $G$ is as defined in Section \ref{subsection:notation}. One can apply Cline, Parshall, Scott, and van der Kallen's \cite{Cline:1977a} rational and generic cohomology results to identify certain cohomology groups for $\Gfq$ with cohomology groups for $G$. Then applying our results on the dimensions of rational cohomology groups, one can obtain corresponding bounds for $G(\Fq)$. For sufficiently large $q$, this approach can be used to recover, and in general, improve upon, the bounds in Theorem \ref{theorem:GH}. The following theorem demonstrates this approach.

\begin{theorem} \label{theorem:H1boundusingG}
Let $r \geq 2$, and set $s = \left[{\frac{r}{2}}\right]$. Assume that $p^{s-1}(p-1) > h$. Then for each finite-dimensional $k\Gfq$-module $V$, one has
\[
\dim \opH^1(G(\Fq),V) \leq \tfrac{1}{h}\dim V.
\]
\end{theorem}

\begin{proof}
Arguing by induction on the composition length, and using the long exact sequence in cohomology, it suffices to assume that $V$ is an irreducible $k\Gfq$-module. Then $V = L(\lambda)$ for some $\lambda \in X_r(T)$. Write $\lambda$ in the form $\lambda = \lambda_0 + p^s\lambda_1$ with $\lambda_0 \in X_s(T)$ and $\lambda_1 \in X(T)_+$, and set $\wt{\lambda} = \lambda_1 + p^{r-s}\lambda_0$. Then $L(\wt{\lambda}) \cong L(\lambda)^{(r-s)}$ as $k\Gfq$-modules. In particular, $\dim L(\lambda) = \dim L(\wt{\lambda})$. Now by \cite[Theorem 5.5]{Bendel:2006}, the stated hypotheses imply that either $\opH^1(G(\Fq),L(\lambda)) \cong \opH^1(G,L(\lambda))$, or $\opH^1(G(\F_q),L(\lambda)) \cong \opH^1(G,L(\wt{\lambda}))$. In either case, the inequality $\opH^1(\Gfq,L(\lambda)) \leq \frac{1}{h} \dim L(\lambda)$ then follows from Theorem \ref{theorem:oneoverh}.
\end{proof}

If $V$ is an irreducible $k\Gfq$-module and if the Weyl group $W$ is of order prime to $p$, then the inequality in Theorem \ref{theorem:H1boundusingG} can be improved by applying Theorem \ref{theorem:highestrootdimension}. 

\subsection{}

For higher degrees, one can make use of recent work of Parshall, Scott, and Stewart \cite{Parshall:2012} to apply the approach of the previous section. Given a positive integer $n$, they show that there exists an integer $r_0$, depending on $n$ and on the underlying root system $\Phi$, such that if $r \geq r_0$, $q = p^r$, and $\lambda \in X_r(T)$, then $\opH^n(G(\Fq),L(\lambda)) \cong \opH^n(G,L(\lambda'))$; see \cite[Theorem 5.8]{Parshall:2012}. Here $\lambda'$ is a certain ``$q$-shift'' of $\lambda$, similar to the weight $\wt{\lambda}$ used in the previous proof. Of importance for our purposes is that $\dim L(\lambda') = \dim L(\lambda)$.  With this, one can recover Theorem \ref{theorem:H1boundusingG} for arbitrary primes, but at the expense of requiring a potentially larger $r$. For degrees 2 and 3, one can use this idea along with Corollary \ref{corollary:boundsonBcoho} to improve, for sufficiently large $r$, upon the bound in Theorem \ref{gkkl}. As above, the proofs of the following two theorems reduce to the case where $V$ is an irreducible $k\Gfq$-module.

\begin{theorem}\label{theorem:H2boundusingG}
There exists a constant $D(\Phi,2)$, depending on $\Phi$, such that if $r \geq D(\Phi,2)$ and if $q = p^r$, then, for each finite-dimensional $k\Gfq$-module $V$, one has
\[
\dim \opH^2(\Gfq,V) \leq \dim V.
\]
\end{theorem}

\begin{theorem}\label{theorem:H3boundusingG}
Suppose $p>h$. Then there exists a constant $D(\Phi,3)$, depending on $\Phi$, such that if $r \geq D(\Phi,3)$ and if $q = p^r$, then, for each finite-dimensional $k\Gfq$-module $V$, one has
\[
\dim \opH^3(\Gfq,V) \leq 2 \cdot \dim V.
\]
\end{theorem}

The constants $D(\Phi,2)$ and $D(\Phi,3)$ in the previous two theorems can be determined recursively. This is done in the proofs of Theorems 5.2 and 5.8 in \cite{Parshall:2012}.

\makeatletter
\renewcommand*{\@biblabel}[1]{\hfill#1.}
\makeatother


\providecommand{\bysame}{\leavevmode\hbox to3em{\hrulefill}\thinspace}
\providecommand{\MR}{\relax\ifhmode\unskip\space\fi MR }
\providecommand{\MRhref}[2]{%
  \href{http://www.ams.org/mathscinet-getitem?mr=#1}{#2}
}
\providecommand{\href}[2]{#2}

\end{document}